\documentclass[11pt]{amsart}
\usepackage{}
\usepackage{mathrsfs}
\usepackage{bbm}
\usepackage{amssymb}
\usepackage{indentfirst}
\usepackage{latexsym,amsfonts,amssymb,amsmath}
\newtheorem{theorem}{Theorem}[section]
\newtheorem{lemma}[theorem]{Lemma}

\newtheorem{question}[theorem]{Question}
\theoremstyle{definition}
\newtheorem{definition}[theorem]{Definition}
\newtheorem{proposition}[theorem]{Proposition}
\theoremstyle{remark}

\begin{document}

\title
{Locally upper bounded poset-valued maps and stratifiable spaces}

\thanks{}

\author{Ying-Ying Jin}
\address{(Y.Y. Jin)
School of Mathematics and Computational Science, Wuyi University, Jiangmen 529020, P.R. China}
\email{yingyjin@163.com}

\author{Li-Hong Xie}
\address{(L.H. Xie) School of Mathematics and Computational Science, Wuyi University, Jiangmen 529020, P.R. China}
\email{yunli198282@126.com}

\author{Han-Biao Yang$^*$}
\address{(H.B. Yang) School of Mathematics and Computational Science, Wuyi University, Jiangmen 529020, P.R. China}
\email{596283897@qq.com}

\thanks{* The corresponding author}

\thanks{$^1$Supported by NSFC (Nos. 11526158, 11601393).}


\begin{abstract}
In this paper, we characterize stratifiable (or semi-stratifiable) spaces, and monotonically countably paracompact (or monotonically countably metacompact) spaces by
expansions of locally upper bounded semi-continuous poset-valued maps.
These extend earlier results for real-valued Locally bounded functions.
\end{abstract}

 \maketitle

\section{Introduction}
Throughout this paper, let $\mathbb{R}$ the set of all real numbers, and $\mathbb{N}$ set of all natural numbers. All topological spaces are assumed to be $T_1-$spaces.

J. Mack characterized \cite{JM} countably paracompact spaces with
locally bounded real-valued functions as follows:

\begin{theorem}(\cite{JM})
A space $X$ is countably paracompact if and only if for each locally bounded
function $h:X \rightarrow\mathbb{R} $ there exists a locally bounded l.s.c. function $g:X \rightarrow \mathbb{R}$ such that $|h|\leq g$.
\end{theorem}

C.R. Borges \cite{Bo} introduced definitions called stratifiable spaces and semi-stratifiable spaces.

\begin{definition}\cite{Bo}
A space $X$ is said to be stratifiable if, to
each open set $U$, one can assign an increasing sequence $(U_n)_{n\in\mathbb{N}}$, called a
stratification of $X$, of open subsets of $X$ such that
\begin{enumerate}
  \item $ \overline{U_n}\subseteq U$ for each $n\in\mathbb{N}$;
  \item $\bigcup_{n\in\mathbb{N}}U_n=U$;
  \item if $ U\subseteq V$, then $  U_n\subseteq V_n$ for each $n\in\mathbb{N}$.
\end{enumerate}
$X$ is said to be semi-stratifiable, if to each open set $U$, one can assign a sequence of closed subsets $(U_n)_{n\in\mathbb{N}}$
such that (2) and (3) above hold.
\end{definition}

Recall that a space $X$ is said to be {\it perfect} \cite{En} if to each open set $U$ of $X$, one can assign an increasing sequence of closed subsets $(U_n)_{n\in\mathbb{N}}$ such that (2) above holds.
A perfect space $X$ is said to be {\it perfectly normal} if $X$ is normal.

It is well known that a space is stratifiable if and only if it is monotonically normal and semi-stratifiable.

C. Good, R. Knight and I. Stares \cite{GK} and C. Pan \cite{Pa} introduced a monotone version of countably paracompact spaces, called monotonically countably paracompact spaces (MCP) and monotonically cp-spaces, respectively, and it was
proved in \cite[Proposition 14]{GK} that both these notions are equivalent.

\begin{definition}\cite{GK}\label{def2.4}
A space $X$ is said to be monotonically countably metacompact (MCM) if there is an operator $U$ assigning to each decreasing sequence $(D_j)_{j\in\mathbb{N}}$
of closed sets with empty intersection, a sequence of open sets $U((D_j))=(U(n,(D_j)))_{n\in\mathbb{N}}$ such that
 \begin{enumerate}
  \item $D_n \subseteq U(n,(D_j))$ for each $n\in\mathbb{N}$;
  \item $\bigcap_{n\in\mathbb{N} }U(n,(D_j))=\emptyset$;
  \item given two decreasing sequences of closed sets $(F_j)_{j\in \mathbb{N}}$ and $(E_j)_{j\in \mathbb{N}}$ such that $F_n \subseteq E_n$ for each $n\in\mathbb{N}$, then $U(n,(F_j))\subseteq U(n,(E_j))$ for each $n\in \mathbb{N}$.
\end{enumerate}
$X$ is said to be monotonically countably paracompact (MCP) if, in addition,

$(2') \bigcap_{n\in \mathbb{N}}\overline{U(n,(D_j))}=\emptyset$.
 \end{definition}

Many insertion results present some classic characterizations of topological spaces, such as stratifiable spaces, monotonically countably paracompact spaces and others.
T. Kubiak \cite{TK} investigated monotonically normal spaces by the monotonization of insertion properties.
P. Nyikos and C. Pan \cite{Pa} and C. Good and I. Stares \cite{GS} respectively gave a characterization of stratifiable spaces by the monotonizations of insertion properties.
Also, C. Good, R. Knight and I. Stares \cite{GK} characterized monotonically countably paracompact spaces by the insertions of semi-continuous functions.

By extending the insertion properties of real-valued maps, K. Yamazaki \cite{KY} introduced the notion of local boundedness for set-valued mappings and described MCP spaces by expansions of locally bounded set-valued mappings.
L.H. Xie, P.F. Yan\cite{XH} gave some characterizations of stratifiable, semi-stratifiable by expansions of set-valued mappings.
K. Yamazaki \cite{Ya}, Y.Y. Jin, L.H. Xie, H.W. Yue \cite{JX} considered the locally upper bounded maps with values in the ordered topological vector spaces and provided new monotone insertion theorems.

The following theorems were proved in \cite[Theorem 2.4]{KY}, \cite[Theorem 3.1 and Theorem 3.2]{Ya} and \cite[Theorem 3.1 and Theorem 3.2]{XH}.

\begin{theorem}(\cite{KY})
For a space $X$, the
following statements are equivalent:

\begin{enumerate}
\item $X$ is MCP (resp. MCM);

\item for every metric space $Y$, there exists a preserved order operator $\Phi$ assigning to each
locally bounded set-valued mapping $\varphi: X \rightarrow \mathcal {B}(Y)$, a locally bounded l.s.c. (resp. a l.s.c.) set-valued mapping $\Phi(\varphi): X \rightarrow \mathcal {B}(Y)$ such that $\varphi\subseteq \Phi(\varphi)$;

\end{enumerate}
where $\mathcal {B}(Y)$ is the set of all nonempty closed bounded sets of $Y$.
\end{theorem}

\begin{theorem}(\cite{Ya})
Let $X$ be a topological space and $Y$ an ordered topological vector space with a positive interior point.
Then, the following conditions are equivalent:
\begin{enumerate}
\item $X$ is MCP (resp. MCM).
\item There exists an operator $\Phi$ assigning to each locally upper bounded map $f:X\rightarrow Y$, a locally upper bounded lower semi-continuous (resp. a lower semi-continuous) map $\Phi(f):X\rightarrow Y$ with $f\leq\Phi(f)$ such that $\Phi(f)\leq\Phi(f')$ whenever $f\leq f'$.

\end{enumerate}
\end{theorem}

\begin{theorem}(\cite{XH})
For a space $X$, the following statements are equivalent:

\begin{enumerate}
\item $X$ is perfectly normal (resp. stratifiable);

\item  for every space $Y$ having a strictly increasing closed cover $\{B_n\}$, there exists an operator $\Phi$ (resp. a preserved order operator $\Phi$) assigning to each
set-valued mapping $\varphi: X \rightarrow \mathcal {F}(Y)$, a l.s.c. set-valued mapping $\Phi(\varphi): X \rightarrow \mathcal {F}(Y)$ such that $\Phi(\varphi)$ is locally bounded at each $x\in U_\varphi$ and that $\varphi\subseteq \Phi(\varphi)$;
\end{enumerate}
\end{theorem}

\begin{theorem}(\cite{XH})
For a space $X$, the following statements are equivalent:

\begin{enumerate}
\item $X$ is perfect (resp. semi-stratifiable);

\item  for every space $Y$ having a strictly increasing closed cover $\{B_n\}$, there exists an operator $\Phi$ (resp. a preserved order operator $\Phi$) assigning to each set-valued mapping $\varphi: X \rightarrow \mathcal {F}(Y)$, a l.s.c. set-valued mapping $\Phi(\varphi): X \rightarrow \mathcal {F}(Y)$ such that $\Phi(\varphi)(x)$ is bounded at each $x\in U_\varphi$ and that $\varphi\subseteq \Phi(\varphi)$;
\end{enumerate}
\end{theorem}

The purpose of this paper is to generalize real-valued locally bounded functions to locally upper bounded maps with values into some bi-bounded complete and bi continuous posets, which are not necessarily vector spaces or spaces with strictly increasing closed covers, by using the way-below relation $\ll$ and the way-above relation $\ll_d$.
This provides some advantage to the real-valued and set-valued cases.
Indeed, the range $\mathbb{R}$ with the total order can be extended to spaces $P$ with the partial order.
Inspired by Theorem 1.2, Theorem 1.3, Theorem 1.4 and Theorem 1.5,
another purpose of this paper is to characterize stratifiable (or semi-stratifiable) spaces, and monotonically countably paracompact (or monotonically countably metacompact) spaces by expansions of locally upper bounded poset-valued maps along the same lines.
At the last part, we also consider monotone
poset-valued insertions on monotonically normal and monotonically countably paracompact spaces.

Throughout this paper,
all the undefined topological concepts can be found in \cite{En}.

\section{Basic facts and definitions}
In this section, some definitions are restated and some basic facts are listed.
Also, some notions are introduced which seem to be convenient
though they may be found in the references.

\begin{lemma}\label{lem2.2}
For a space $X$, the following statements are equivalent:
\begin{enumerate}
\item $X$ is semi-stratifiable (resp. stratifiable);

\item there is an operator $F$
assigning to each increasing sequence of open sets $(U_{j})_{j\in
\mathbb{N}}$, an increasing sequence of closed sets
$(F(n,(U_{j})))_{n\in\mathbb{N}}$ such that
\begin{enumerate}
\item [(i)] $U_{n}\supseteq
F(n,(U_{j}))$ for each $n\in\mathbb{N}$;

\item [(ii)] $\bigcup_{n\in
\mathbb{N}}F(n,(U_{j}))=\bigcup_{n\in\mathbb{N}}U_{n}$ (resp. (ii)' $\bigcup_{n\in
\mathbb{N}}Int F(n,(U_{j}))=\bigcup_{n\in\mathbb{N}}U_{n}$);

\item [(iii)] given two increasing sequences of open sets
$(U_{j})_{j\in\mathbb{N}}$ and $(G_{j})_{j\in\mathbb{N}}$ such that
$U_{n}\subseteq G_{n}$ for each $n\in\mathbb{N}$, then
$F(n,(U_{j}))\subseteq F(n,(G_{j}))$ for each
$n\in\mathbb{N}$.
\end{enumerate}
\end{enumerate}
\end{lemma}

\begin{proof}
From De Morgan's laws it follows easily that conditions (2) in Theorem 2.2 and (2) in \cite[Theorem 3.6, 3.7]{XY} are equivalent.
\end{proof}

It is well known that semi-stratiffiable (stratiffiable) spaces
are naturally monotone versions of perfect (perfectly normal) spaces.
We can easily obtain the following result without proof.

\begin{lemma}\label{the2.4}
For a space $X$, the following statements are equivalent:
\begin{enumerate}
\item $X$ is perfect (resp. perfectly normal);

\item there is an operator $F$
assigning to each increasing sequence of open sets $(U_{j})_{j\in
\mathbb{N}}$, an increasing sequence of closed sets
$(F(n,(U_{j})))_{n\in\mathbb{N}}$ such that
\begin{enumerate}
\item [(i)] $U_{n}\supseteq
F(n,(U_{j}))$ for each $n\in\mathbb{N}$;

\item [(ii)] $\bigcup_{n\in
\mathbb{N}}F(n,(U_{j}))=\bigcup_{n\in\mathbb{N}}U_{n}$ (resp. $\bigcup_{n\in
\mathbb{N}}Int F(n,(U_{j}))=\bigcup_{n\in\mathbb{N}}U_{n}$);
\end{enumerate}
\end{enumerate}
\end{lemma}

\begin{lemma}\label{lem2.5}
A space $X$ is said to be monotonically countably metacompact (MCM) (resp. monotonically countably paracompact (MCP)) if and only if there is an operator $F$ assigning to each increasing sequence $(U_j)_{j\in\mathbb{N}}$
of open sets of $X$ satisfying $\bigcup_{i\in \mathbb{N}}U_i=X$, a sequence of closed sets $F((U_j))=(F(n,(U_j)))_{n\in\mathbb{N}}$ such that
 \begin{enumerate}
  \item $U_n \supseteq F(n,(U_j))$ for each $n\in\mathbb{N}$;
  \item $\bigcup_{n\in\mathbb{N} }F(n,(U_j))=X$ (resp. (2)' $\bigcup_{n\in \mathbb{N}}Int F(n,(U_j))=X$);
  \item given two increasing sequences of open sets $(U_j)_{j\in \mathbb{N}}$ and $(G_j)_{j\in \mathbb{N}}$ such that $U_n \subseteq G_n$ for each $n\in\mathbb{N}$, then $F(n,(U_j))\subseteq F(n,(G_j))$ for each $n\in \mathbb{N}$.
\end{enumerate}
\end{lemma}

 \begin{proof}
From De Morgan's laws it follows easily that conditions (1), (2) and (3) are equivalent to Definition 2.4.
\end{proof}

In the rest of this section, let us recall some definitions and terminology from \cite{Gi, Ke}.

Let $P=(P,\leq)$ be a poset.
For $a, b \in P$, the symbol $[a, b]$ stands for $\{y\in P:a \leq y\leq b\}$.
A subset $A$ of $P$ is said to be {\it directed} (resp. {\it filtered}) if $A$ is nonempty and for every $x, y\in A$ there exists $z\in A$ such that $x\leq z$ and $y\leq z$ (resp. $z\leq x$ and $z\leq y$).
For a subset $A$ of $P$, $\bigvee A$(resp. $\bigwedge A$) stands for the sup (resp. inf) of $A$, if exists.
For $x, y\in P$, $x$ is {\it way below} $y$, in symbol $x \ll y$, if for all directed subset $D$ of $P$ with $\bigvee D$, the relation $y\leq\bigvee D$ always implies the existence of $d\in D$ with $x \leq d$.
For $x, y\in P$, $x$ is {\it way above} $y$, in symbol $y\ll_d x$, if for all filtered subset $F$ of $P$ with $\bigwedge F$, the relation $\bigwedge F\leq y$ always implies the existence of $f\in F$ with $f\leq x$.
In a lattice $L$, $x \ll y$ (resp. $x \ll_d y$) if and only if for every subset $A$ of $L$ with $\bigvee A$ (resp. $\bigwedge A$) the relation $y\leq \bigvee A$ (resp. $\bigwedge A \leq y$) always implies the existence of a finite subset $B$ of $A$ such that $x \leq \bigvee B$ ($\bigwedge B\leq x$).
Note that the way-above relation $\ll_d$ is precisely the dual relation of the way below relation of $P^{op}$, i.e. $x \ll_d y\Leftrightarrow y\ll^{op}x$.
An element $x$ of $P$ is {\it isolated from below} (resp. {\it isolated from above}) if $x \ll x$ (resp. $x \ll_dx$).
Clearly, an element $x$ of $P$ is isolated from above iff $x$ is isolated from below in $P^{op}$.
On $\mathbb{R}$, it is clear that $x \ll y$ if and only if $x <y$ if and only if $x \ll_dy$.
On the unit interval $[0, 1]$ of $\mathbb{R}$, we have that $x \ll y$ if and only if $x <y$ or $x =y=0$, and that $x \ll_dy$ if and only if $x <y$ or $x =y=1$.
A poset $P$ is {\it continuous} if $\{u \in P:u \ll x\}$ is directed and $x =\bigvee\{u \in P:u\ll x\}$ for all $x \in P$.
A poset $P$ is {\it dually continuous} if $\{u\in P:x \ll_du\}$ is filtered and $x =\bigwedge\{u \in P:x \ll_du\}$ for all $x \in P$, in other words, if $P^{op}$ is continuous.
A poset $P$ is called bicontinuous if $P$ is continuous and dually continuous.
As is pointed out in \cite{Ke}, the way-below relation in a bicontinuous poset need not be the opposite of the way-above relation in the sense that $x \ll y$ does not imply $x \ll_dy$.
A {\it bicontinuous lattice} $P$ is a lattice which is bicontinuous as a poset.
It should be noted that we do not require the completeness of $P$ in our definition of $P$ being bicontinuous, where a {\it complete} lattice $P$ is a poset in which every subset has the sup and the inf.

We call a poset $P$ {\it lower-bounded complete} (resp. {\it upper-bounded complete}) if every non-empty subset $A$ of $P$ with a lower bound (resp. an upper bound) has the inf (resp. sup).
When $P$ is lower-bounded complete and upper-bounded complete, we call $P$ {\it bi-bounded complete}.
Note that every bounded complete domain in the sense of [5] is bi-bounded complete.
For maps $f, g:X\rightarrow P$ into a poset $P$, the symbol $f\ll g$ (resp. $f\ll_d g$, $f\leq g$) stands for $f(x)\ll g(x)$ (resp. $f(x)\ll_d g(x)$, $f(x) \leq g(x)$) for each $x\in X$.
For a point $z$ and a pair of points $\langle y,y'\rangle$ of a poset $P$, $z$ is an {\it interpolated point} of $\langle y,y'\rangle$ if $y\ll_d z\ll y'$.
A pair $\langle f,g\rangle$ of maps $f, g:X\rightarrow P$ has {\it interpolated points pointwise} if $\langle f(x), g(x)\rangle$ has an interpolated point for each $x\in X$.

For a subset $B$ of a poset $P$ and $y, y'\in P$, the pair of points $\langle y,y'\rangle$ has interpolated points on $B$ if there exists $z\in B$ such that $z$
is an interpolated point of $\langle y,y'\rangle$.
A pair $\langle f,g\rangle$ of maps $f, g:X\rightarrow P$ has interpolated points pointwise on B if $\langle f(x), g(x)\rangle$ has interpolated points on $B$ for each $x\in X$.

For a non-empty subset $A$ of $X$, a pair $\langle f,y\rangle$ (resp. $\langle y,f\rangle$) of a map $f: X\rightarrow P$ and a point $y\in P$ has {\it interpolated points of $A$} if $\langle f(x),y\rangle$ (resp. $\langle y,f(x)\rangle$) has interpolated points for each $x\in A$.

For a non-empty subset $A$ of $X$, a pair $\langle f,g\rangle$ of maps $f,g:X\rightarrow P$ has {\it interpolated points of $A$} if $\langle f(x),g(x)\rangle$ has interpolated points for each $x\in A$.
We define $G_{f,g}=\{x\in X: \langle f(x),g(x)\rangle \text{~has an interpolated point~}\}.$

See \cite{En} and \cite{Gi} for undefined terminology.

\begin{lemma}\label{lem2.6}\cite{Gi}
For a poset $P$, the following statements hold.
 \begin{enumerate}
  \item $x \ll_dy$$\Rightarrow$$x \leq y$;
  \item $u \leq x \ll_dy\leq v\Rightarrow u \ll_dv$;
  \item $z\ll_dx$ and $z\ll_dy$$\Rightarrow$$z\ll_dx \wedge y$, whenever $x \wedge y$ exists.
\end{enumerate}
If $P$ is a dually continuous poset, (4) below also holds.

 (4) $x\ll_dy$ $\Rightarrow$ $\exists$ $z\in P$ s.t. $x \ll_dz\ll_dy$.
\end{lemma}

For a subset $A$ of a topological space $X$ and $x\in X$, $\overline{A}$ stands for the closure of $A$ and $\mathcal{N}_x$ is the set of all neighborhoods of $x$.

Let $f:X\rightarrow P$ be a map from a topological space $X$ to a poset $P$ (which is not assumed to be a topological space), and $x \in X$. Set
$$\mathcal{N}_{x*}(f) = \{N \in \mathcal{N}_x : f(N) \text{~has the inf}\}, \text{~and~} \mathcal{N}^*_x (f) = \{N \in \mathcal{N}_x : f(N) \text{~has the sup}\}.$$
We call that $f$ {\it admits} $f_*(x)$ if $\mathcal{N}_{x*}(f)\neq\emptyset$ and $\{\bigwedge f(N):N\in\mathcal{N}_{x*}(f)\}$ has the sup, and then we define $f_*(x)=\bigvee\{\bigwedge f(N):N\in\mathcal{N}_{x*}(f)\}$.
Also, $f$ {\it admits} $f^*(x)$ if the set $\mathcal{N}^*_x(f)\neq\emptyset$ and $\{\bigvee f(N) :N\in \mathcal{N}_x^*(f)\}$ has the inf, and then we define
$f^*(x) =\bigwedge\{\bigvee f(N):N\in \mathcal{N}_{x*}(f)\}$.
A map $f:X\rightarrow P$ is {\it lower semi-continuous} (resp. {\it upper semi-continuous}) at $x$ if $f$ admits $f_*(x)$ (resp. $f^*(x)$) and $f(x) =f_*(x)$ (resp. $f(x) =f^*(x)$).
A map $f:X\rightarrow P$ is {\it lower semi-continuous} (resp. {\it upper semi-continuous}) if $f$ is lower (resp. upper) semi-continuous at every $x\in X$.
Since $P$ is not assumed to be complete, for $A \subset P$, $\bigwedge A$ or $\bigvee A$ does not necessarily exist.
If there is no confusion, we simply express $f_*(x) =\bigvee_{N\in\mathcal{N}_x}\bigwedge f(N)$ and $f^*(x) =\bigwedge _{N\in \mathcal{N}_x}\bigvee f(N)$ for each $x \in X$ (if exists).

It is defined in \cite{En} that a real-valued function $f:X\rightarrow \mathbb{R} $ is lower semi-continuous if $\{x:f(x)>r\}$ is open for each $r\in \mathbb{R}$
(namely, for each $x\in X$ and each $\varepsilon>0$ there exists a neighborhood $O_x$ of $x$ such that $f(x')>f(x)-\varepsilon$ for each $x'\in O_x$).
A real-valued function $f:X\rightarrow\mathbb{R}$ is upper semi-continuous if $-f$ is lower semi-continuous.
Note that this definition coincide with the above definition of semi-continuous maps with values into ordered topological vector spaces $Y$ for $Y=\mathbb{R}$.

\begin{proposition}\cite{Yama}\label{prop2.7}
Let $P$ be a poset, $x\in X$ and $f:X\rightarrow P$ a map. Consider the following conditions:
\begin{enumerate}
\item [(1)] $\{\bigwedge f(N):N\in\mathcal{N}_{x*}(f)\}$ (resp.) $\{\bigvee f(N):N\in\mathcal{N}_{x}^*(f)\}$ has the sup (resp. inf);
\item [(2)] $\mathcal{N}_{x*}(f)\neq\emptyset$ (resp. $\mathcal{N}_{x}^*(f)\neq\emptyset$ );
\item [(3)] $\mathcal{N}_{x*}(f)$ (resp. $\mathcal{N}_{x}^*(f)$) is a neighborhood base of $x$.
\end{enumerate}
Then, the following statements (a), (b), (c) and (d) hold.

\begin{enumerate}
\item [(a)] If P is lower-bounded (resp. upper-bounded) complete, (1)$\Rightarrow$(2) and (2)$\Rightarrow$(3) hold.
\item [(b)] If P is a cdcpo (resp. cfcpo), (3)$\Rightarrow$(1) holds.
\item [(c)] If $P$ is bi-bounded complete, the conditions (1), (2)and (3)are equivalent, that is, $f$ admits $f_*(x)$ (resp.$f^*(x)$) whenever either one of (1), (2 )and (3)
holds
\item [(d)] If $P$ is bi-bounded complete and $f(X)$ has a lower (resp. an upper) bound, (2) holds, thus, $f$ admits $f_*$ (resp. $f^*$).
\end{enumerate}
\end{proposition}

\begin{proposition}\cite{Yama}\label{prop2.8}
Let $P$ be a poset, $x\in X$ and $f:X\rightarrow P$ a map. Consider the following conditions:
Let $X$ be a topological space, $P$ a poset, and assume that $f$ admits $f_*$ and $\mathcal{N}_{x*}(f)$ is a
neighborhood base of $x$ for each $x\in X$. For a map $f: X\rightarrow P$, consider the following conditions:
\begin{enumerate}
\item [(1)] f is lower semi-continuous;
\item [(2)] $\{x\in X: a\ll f(x)\}$ is open for each $a\in P$;
\item [(3)] $\{x\in X: f(x)\leq a\}$ is closed for each $a\in P$.
\end{enumerate}
Then, $(1)\Rightarrow(3)$ always holds. If P is continuous, $(1)\Leftrightarrow(2)$ holds.
\end{proposition}

\begin{proposition}\cite{Yama}\label{prop2.9}
For a topological space $X$ and a bi-bounded complete, continuous (resp. dually continuous) poset $P$,
if $f:X\rightarrow P$ is lower (resp. upper)semi-continuous, then $\{x\in X: a \ll f(x)\}$ (resp. $\{x\in X: f(x) \ll_d a\}$ )is open for each $a\in P$.
The converse holds if, in addition, $\mathcal{N}_{x*}(f)\neq\emptyset$ (resp. $\mathcal{N}_{x}^*(f)\neq\emptyset$).
\end{proposition}

We call a point $z_0$ of a  poset $P$ is
a $\ll_d$-{\it increasing} $\ll-${\it limit point} \cite{Yama} if there exists a sequence $\{y_i:i\in\mathbb{N}\}\subset P$ such that
$y_i\ll_d y_{i+1}$, $y_i\ll y_0 (i\in \mathbb{N})$ and $y_0=\bigvee_{i\in\mathbb{N}}y_i$.
Also, $y_0$ of $P$ is a $\ll-${\it decreasing} $\ll_d-$ {\it limit point} if there exist $y_i\in P (i\in\mathbb{N})$ such that $y_{i+1}\ll y_i$, $y_0\ll_d y_i (i\in\mathbb{N})$ and $y_0=\bigwedge_{i\in\mathbb{N}}y_i$.

We introduce the following:

\begin{definition}\cite{Yama}
For a poset $(P,\leq)$ and $y, z\in P$, $y$ is a lower bound (resp. upper bound) of $f:X\rightarrow P$ if $y$ is a lower (resp. upper) bound of $f(X)$.
\end{definition}

\begin{definition}
For a topological space $X$ and a bi-bounded complete, dually continuous poset $P$ with a $\ll_{d}$-increasing $\ll$-limit point $y_0$, a map $f:X\rightarrow P$ is called locally upper bounded about $(y_i)_{i\in\mathbb{N}}$ if for every $x\in X$ there exist $n\in\mathbb{N}$ and a neighborhood $O_x$ of $x$ such that $f(x')\leq y_n$ for each $x'\in O_x$.
\end{definition}

For a mapping $g:$ $X\rightarrow P$, define$$U_g=\{x\in X: g(x) \text{~is locally upper bounded about $(y_i)_{i\in\mathbb{N}}$ ~} \}.$$
Clearly, $U_g$ is an open set in $X$.
$$F_g=\{x\in X: \exists n\in\mathbb{N} \text{~such that~} g(x)\leq y_n\}.$$

\section{main results}

\begin{theorem}\label{the3.1}
Let $P$ be a bi-bounded complete, continuous poset $P$ with a $\ll_{d}$-increasing $\ll$-limit point $y_0$.
Then $X$ is stratifiable if and only if
there exists an operators $\Phi$ assigning to each map $g:X\rightarrow P$ with $F_g\neq\emptyset$,
a l.s.c. map $\Phi(g):X\rightarrow P$ such that
$g\leq \Phi(g)$, $\Phi(g)$ is locally upper bounded about $(y_i)_{i\in\mathbb{N}}$ at each $x\in U_g$ and that $\Phi(g)\leq \Phi(g')$ whenever $g\leq g'$.
\end{theorem}

\begin{proof}
Assume that $X$ is a stratifiable.
There exists an operator $F$ satisfying (i), (ii)' and (iii) in Lemma \ref{lem2.2}.
Let $P$ be a bi-bounded complete, continuous poset $P$ with a $\ll_{d}$-increasing $\ll$-limit point $y_0$.
For each map $g:X\rightarrow P$ with $F_g\neq\emptyset$ and each $n\in \mathbb{N}$, define
$$ U_n(g)=Int\{x\in X: g(x)\leq y_n\}  \quad\quad \quad\quad(\ref{the3.1}.1)$$
Then we have $U_g=\bigcup_{n\in \mathbb{N}}U_n(g).$
In fact, for each $x\in U_g$, then there exists an open neighborhood $O$ of $x$ such that $g(x')\leq y_i $ for some $i\in \mathbb{N}$ and each $x'\in O$, which implies that $ x\in U_i(g)$. It implies that $U_g\subseteq\bigcup_{n\in \mathbb{N}}U_n(g).$
On the other hand, take any $x\in \bigcup_{n\in \mathbb{N}}U_n(g)$.
Then there is $U_j(g)$ such that $x\in U_j(g)$, and therefore, there exists an open neighborhood $O$ of $x$ such that $ g(x')\leq y_j$ for each $x'\in O$.
It implies that $x\in O\subseteq U_g$.

Hence, $F((U_j(g)))=(F(n,(U_j(g))))_{n\in\mathbb{N}}$ is a sequence of closed subsets of $X$ such that
$$F(n,(U_j(g)))\subset U_n(g) \text{~for each~} n\in\mathbb{N};\quad\quad \quad\quad(\ref{the3.1}.2)$$
$$\bigcup_{n\in\mathbb{N}}Int F(n, (U_j(g)))=\bigcup_{n\in \mathbb{N}}U_n(g);\quad\quad \quad\quad(\ref{the3.1}.3)$$
$$F(n,(U_j(g)))\subset F(n+1,(U_j(g))), n\in\mathbb{N};\quad\quad \quad\quad(\ref{the3.1}.4)$$
Thus, we can define $\Phi(g):X\rightarrow P$ as follows:

$$\Phi(g)(x)=\left\{
\begin{array}{rcl}
y_1    &      & {x\in F(1,(U_j(g)))}\\
y_{n+1}    &      & x\in F(n+1,(U_j(g)))\setminus F(n,(U_j(g)))\\
y_0    &      &x\in X\setminus \bigcup_{n\in\mathbb{N}}F(n, (U_j(g)))\\
\end{array} \right. \quad\quad \quad\quad(\ref{the3.1}.5)$$
It is obvious that $\Phi(g)(x)\leq y_0$ and $\Phi(g)$ has a lower bounded $y_1$.

To show $g\leq\Phi(g)$. For each $x\in X\setminus U_g$, $g(x)\leq y_0=\Phi(g)(x)$ is obvious. Let $x\in U_\varphi $.
Then, $\Phi(\varphi)(x)=y_{i}$ for some $i\in \mathbb{N}$, and therefore, $x\in F(i,(U_j(g)))\setminus F(i-1,(U_j(g)))$.
Since $g(x)\leq y_{i}=\Phi(g)(x)$ by $x\in F(i,(U_j(g)))\subset U_{i}(g)=Int\{x\in X: g(x)\leq y_{i}\}$.
If $x\in F(1,(U_j(g)))$, it is obvious that $g(x)\leq y_{1}=\Phi(g)(x)$. Thus, we have $g\leq\Phi(g)$.

To show that  $\Phi(g):X\rightarrow P$ is locally upper bounded about $(y_i)_{i\in\mathbb{N}}$ at each $x\in U_g$. Take any $x\in U_g$,
by (\ref{the3.1}.3), there exists $n\in\mathbb{N}$ such that $x\in Int F(n+1, (U_j(g)))\setminus Int F(n, (U_j(g)))$.
Consider the neighborhood $O_x=Int F(n+1, (U_j(g)))$ of $x$.
For each $x'\in O_x$, it follows from the definition of $\Phi(g)$ that $\Phi(g)(x')\leq y_{n+1}$.
This completes the
proof that $\Phi(g)$ is locally upper bounded about $(y_i)_{i\in\mathbb{N}}$ at each $x\in U_g$.

Next we show $\Phi(g)$ is l.s.c..
For each $x\in \bigcup_{n\in\mathbb{N}} F(n, (U_j(g)))$, there exists some $m\in\mathbb{N}$ such that $x\in F(m+1, (U_j(g)))\setminus F(m, (U_j(g)))$.
We consider the neighborhood $O_x=X\setminus F(m, (U_j(g)))$ of $x$.
For each $x'\in O_x$, we can get $\Phi(g)(x')\geq y_{m+1}=\Phi(g)(x)$.
Therefore, $\bigwedge\Phi(g)(O_x)=y_{m+1}$ and $O_x\in \mathcal{N}_{x*}(\Phi(g))$.
This provides that $\Phi(g)(x)$ admits $\Phi(g)_*(x)$ because of (c) of Proposition \ref{prop2.7}.
We have
$$\Phi(g)_*(x)=\bigvee_{N\in\mathcal{N}_{x*}(\Phi(g))}\bigwedge\Phi(g)(N)\geq\bigwedge_{x'\in O_x}\Phi(g)(x')=y_{m+1}=\Phi(g)(x).$$
Hence, $\Phi(g)_*(x)=\Phi(g)(x)$ for each $x\in X$, that is, $\Phi(g)(x)$ is l.s.c. at $x$.
For each $x\in X\setminus\bigcup_{n\in\mathbb{N}} F(n, (U_j(g)))$, we consider the neighborhood $V=X\setminus\bigcup_{n\in\mathbb{N}} F(n, (U_j(g)))$ of $x$.
For each $x'\in V$, we can get $\Phi(g)(x')= y_{0}=\Phi(g)(x)$.
Therefore, $\bigwedge\Phi(g)(V)=y_{o}$ and $V\in \mathcal{N}_{x*}(\Phi(g))$.
This provides that $\Phi(g)(x)$ admits $\Phi(g)_*(x)$ because of (c) of Proposition \ref{prop2.7}.
We have
$$\Phi(g)_*(x)=\bigvee_{N\in\mathcal{N}_{x*}(\Phi(g))}\bigwedge\Phi(g)(N)\geq\bigwedge_{x'\in O_x}\Phi(g)(x')=y_{0}=\Phi(g)(x).$$
Hence, $\Phi(g)_*(x)=\Phi(g)(x)$ for each $x\in X$, that is, $\Phi(g)(x)$ is l.s.c. .

Finally, let $g':X\rightarrow P$ be a map with $g\leq g'$.
Then
$$\{x\in X: g'(x)\leq y_n\}\subseteq \{x\in X: g(x) \leq y_n\}$$
and hence, $U_n(g)\supseteq U_n(g')$ for each $n\in\mathbb{N}$.
Therefore we have $F(n,(U_j(g)))\supseteq F(n,(U_j(g')))$ for each $n\in\mathbb{N}$, and

$$\bigcup_{n\in\mathbb{N}}Int F(n, (U_j(g)))=\bigcup_{n\in \mathbb{N}}U_n(g)$$
$$\bigcup_{n\in\mathbb{N}}Int F(n, (U_j(g')))=\bigcup_{n\in \mathbb{N}}U_n(g')$$

For each $x\in \bigcup_{n\in \mathbb{N}}U_n(g')$,
there exists $n\in \mathbb{N}$ such that $x\in F(n+1, (U_j(g')))\setminus F(n, (U_j(g')))$.
That is $x\in F(n+1, (U_j(g')))\subseteq F(n+1, (U_j(g)))$.
By (\ref{the3.1}.5), we can get $\Phi(g)(x)\leq y_{n+1}=\Phi(g')(x)$.
$\Phi(g)(x)\leq y_0= \Phi(g')(x)$ is obvious whenever $x\in X\setminus\bigcup_{n\in \mathbb{N}}U_n(g')$,
which proves the necessity.

Conversely, let $(U_j)_{j\in\mathbb{N}}$ be a sequence of increasing open subsets of $X$.
Define a map $g_{((U_j))}:X\rightarrow P$ by:
$$g_{((U_j))}(x)=\left\{
\begin{array}{rcl}
y_1     &      & {x\in U_{1}}\\
y_{n+1}    &      & {x\in U_{n+1}\backslash \texttt{}U_{n}}\\
y_0   &      & {x\in X\setminus\bigcup_{n\in\mathbb{N}}U_n}\\
\end{array} \right. \quad\quad \quad\quad(\ref{the3.1}.6)$$
Then, we have $U_{g_{((U_j))}}=\bigcup_{n\in\mathbb{N}}U_n$, where
$$U_{g_{((U_j))}}=\{x\in X: g_{((U_j))} \text{~is locally upper bounded about $(y_i)_{i\in\mathbb{N}}$ at~}x \}.$$
By the assumption, there exist an operators $\Phi$ assigning to each $g_{((U_j))}$ with an upper bound,
a l.s.c. map $\Phi(g_{((U_j))}):X\rightarrow P$ such that
$g_{((U_j))}\leq \Phi(g_{((U_j))})$, $\Phi(g_{((U_j))})(x)$ is locally upper bounded about $(y_i)_{i\in\mathbb{N}}$ at each $x\in$ and $\Phi(g)\leq \Phi(g')$ whenever $g\leq g'$.
For each sequence $(U_j)_{j\in\mathbb{N}}$ of increasing open subsets of $X$, define

$$F(n,(U_j))=\{x\in X: \Phi(g_{((U_j))})(x)\leq y_n\}\quad\quad \quad\quad(\ref{the3.1}.7)$$
We can get that $F(n,(U_j))$ is closed, by (a) of Proposition \ref{prop2.7} and Proposition \ref{prop2.8}.
It suffices to show the operator $F$ satisfies (i), (ii)' and (iii) of Lemma \ref{lem2.2}.

To see $U_n \supseteq F(n,(U_j))$ for each $n\in\mathbb{N}$ and let $x\in F(n,(U_j))$.
Then, $g_{((U_j))}(x)\leq\Phi(g_{((U_j))})(x)\leq y_n$ and $g_{((U_j))}(x)\leq y_n$, and thus $x\in U_{m}\setminus U_{m-1}$ and $m\leq n$ by (\ref{the3.1}.1).
So we have $x\in U_m\subset U_n$.
Hence $U_n \supseteq F(n,(U_j))$ holds.
In additional, $\Phi(g_{((U_j))})$ is l.s.c., so $F(n,(U_j))$ is a closed set of $X$ for each $n\in \mathbb{N}$. This shows that the condition (i) is satisfied.

To show (2)', note that $\Phi(g_{((U_j))})$ is locally upper bounded about $(y_i)_{i\in\mathbb{N}}$ at each $x\in U_{g_{((U_j))}}$.
Then, for each $x\in U_{g_{((U_j))}}$, there exists an open neighborhood $O$ of $x$ such that $\Phi(g_{((U_j))})(x')\leq y_{n_0}$ for some $n_0\in \mathbb{N}$ and each $x'\in O$.
It implies that $x\in IntF(n,(U_j))$.
Hence, $\bigcup_{n\in\mathbb{N}}Int F(n, (U_j(g)))=U_{g_{((U_j))}}=\bigcup_{n\in \mathbb{N}}U_n(g)$.

To show (3),
let $((G_j))$ be an increasing sequence of open subsets of $X$ such that $(U_j)\preceq (G_j)$.
Since $U_n\subseteq G_n$ for each $n\in \mathbb{N}$, it follows from (\ref{the3.1}.6) that $g_{((G_j))}(x)\leq g_{((U_j))}(x)$.
Hence, we have $\Phi(g_{((G_j))})\leq\Phi(g_{((U_j))})$.
Furthermore,
 $F(n,(U_j))=\{x\in X: \Phi(g_{((U_j))})(x)\leq y_n\}\subseteq \{x\in X: \Phi(g_{((G_j))})(x)\leq y_n\}=F(n,(G_j))$ for each $n\in\mathbb{N}$,
which implies that $F((U_j))\preceq F((G_j))$.
Thus, $X$ is a stratifiable space.
\end{proof}

\begin{theorem}\label{the3.2}
Let $P$ be a bi-bounded complete, continuous poset $P$ with a $\ll_{d}$-increasing $\ll$-limit point $y_0$.
Then $X$ is semi-stratifiable if and only if
there exists an operators $\Phi$ assigning to each map $g:X\rightarrow P$ with an upper bound $y_0$,
a l.s.c. map $\Phi(g):X\rightarrow P$ such that
$g\leq \Phi(g)$, $\Phi(g)$ is upper bounded at each $x\in U_g$ and that $\Phi(g)\leq \Phi(g')$ whenever $g\leq g'$.
\end{theorem}

\begin{proof}
Assume that $X$ is a semi-stratifiable.
There exists an operator $F$ satisfying (i), (ii) and (iii) in Lemma \ref{lem2.2}.
Let $P$ be a bi-bounded complete, continuous poset $P$ with a $\ll_{d}$-increasing $\ll$-limit point $y_0$.
For each map $g:X\rightarrow P$ and each $n\in \mathbb{N}$, define $U_n(g)$ as (\ref{the3.1}.1) in Theorem \ref{the3.1}.

Then we have $U_g=\bigcup_{n\in \mathbb{N}}U_n(g).$
Define $\Phi(g):X \rightarrow P$ as (\ref{the3.1}.5) in Theorem \ref{the3.1}.

We only show that $\Phi(g)(x)$ is upper bounded at each $x\in U_g$, since the other properties of $\Phi(g)$ are proved in Theorem \ref{the3.1}.

Take any $x\in U_g$, by (ii) of lemma \ref{lem2.2}, we have $x\in \bigcup_{n\in \mathbb{N}}U_n(g)=\bigcup_{n\in\mathbb{N}}F(n, (U_j(g)))$, and therefore, there exists $k\in \mathbb{N}$ such that $x\in F(k, (U_j(g)))\setminus F(k-1, (U_j(g)))$.
It implies that $\Phi(g)(x)=y_k$.
This completes the
proof that $\Phi(g)(x)$ is upper bounded at each $x\in U_g$.

Conversely, let $(U_j)_{j\in\mathbb{N}}$ be a sequence of increasing open subsets of $X$.
Define a map $g_{((U_j))}:X\rightarrow P$ as (\ref{the3.1}.6) in Theorem \ref{the3.1}.

Then, we have $U_{g_{((U_j))}}=\bigcup_{n\in\mathbb{N}}U_n$, where
$$U_{g_{((U_j))}}=\{x\in X: g_{((U_j))} \text{~is locally upper bounded at~}x \}.$$
By the assumption, there exist an operators $\Phi$ assigning to each $g_{((U_j))}$ with an upper bound,
a l.s.c. map $\Phi(g_{((U_j))}):X\rightarrow P$ such that $\Phi(g_{((U_j))})$ is bounded at each $x\in U_{g_{((U_j))}}$,
$g_{((U_j))}\leq \Phi(g_{((U_j))})$, and that $\Phi(g)\leq \Phi(g')$ whenever $g\leq g'$.
For each sequence $(U_j)_{j\in\mathbb{N}}$ of increasing open subsets of $X$ and each $n\in \mathbb{N}$, define the operator $F$ as (\ref{the3.1}.7) in Theorem \ref{the3.1}.

It suffices to show that the operator $F$ satisfies (i), (ii) and (iii) of Lemma \ref{lem2.2}.
We can get that $F(n,(U_j))$ is closed, by (a) of Proposition \ref{prop2.7} and Proposition \ref{prop2.8}.
It suffices to show the operator $F$ satisfies (i), (ii) and (iii) of Lemma \ref{lem2.2}.
The proof that the operator $F$ satisfies (i) and (iii) of Lemma \ref{lem2.2} is as same as Theorem \ref{the3.1}, so we only shows that the operator $F$ satisfies (ii) of Lemma \ref{lem2.2}.

To show (ii), note that $\Phi(g_{((U_j))})$ is upper bounded at each $x\in U_{g_{((U_j))}}$.
Then, for each $x\in U_{g_{((U_j))}}$, there exists $n_0\in \mathbb{N}$ such that $\Phi(g_{((U_j))})(x)\leq y_{n_0}$.
It implies that $x\in F(n_0,(U_j))$.
Hence, $\bigcup_{n\in\mathbb{N}} F(n, (U_j(g)))=U_{g_{((U_j))}}=\bigcup_{n\in \mathbb{N}}U_n(g)$.
Thus, $X$ is a semi-stratifiable space.
\end{proof}

\begin{theorem}\label{the3.3}
Let $P$ be a bi-bounded complete, dually continuous poset $P$ with a $\ll_{d}$-increasing $\ll$-limit point $y_0$.
Then $X$ is MCP if and only if
there exists an operators $\Phi$ assigning to each map $g:X\rightarrow P$ which is locally upper bounded about $(y_i)_{i\in\mathbb{N}}$ with a lower bound,
a l.s.c. map $\Phi(g):X\rightarrow P$ which is locally upper bounded about $(y_i)_{i\in\mathbb{N}}$ such that
$g\leq \Phi(g)$, and $\Phi(g)\leq \Phi(g')$ whenever $g\leq g'$.
\end{theorem}

\begin{proof}
Suppose that $X$ is MCP and $F$ is any operator that satisfies conditions (1), (2)' and (3) of Lemma \ref{lem2.5}.
Let $g:X\rightarrow P$ be a locally upper bounded map.
For each $n\in \mathbb{N}$, we define
$$ U_n(g)=Int\{x\in X: g(x)\leq y_n\}  \quad\quad \quad\quad(\ref{the3.3}.1)$$
Then, $\{U_n(g):n\in\mathbb{N}\}$ is a increasing sequence of open subsets of $X$ because of Proposition \ref{prop2.9}.
It is clear that $\bigcup_{n\in \mathbb{N}}U_n(g)=X$.

Hence, $F((U_j(g)))=(F(n,(U_j(g))))_{n\in\mathbb{N}}$ is a sequence of closed subsets of $X$ such that
$$F(n,(U_j(g)))\subset U_n(g) \text{~for each~} n\in\mathbb{N};\quad\quad \quad\quad(\ref{the3.3}.2)$$
$$\bigcup_{n\in\mathbb{N}}Int F(n, (U_j(g)))=X;\quad\quad \quad\quad(\ref{the3.3}.3)$$
$$F(n,(U_j(g)))\subset F(n+1,(U_j(g))), n\in\mathbb{N};\quad\quad \quad\quad(\ref{the3.3}.4)$$
Thus, we can define $\Phi(f):X\rightarrow P$ as follows:

$$\Phi(g)(x)=\left\{
\begin{array}{rcl}
y_1    &      & {x\in F(1,(U_j(g)))}\\
y_{n+1}    &      & x\in F(n+1,(U_j(g)))\setminus F(n,(U_j(g)))\\
\end{array} \right. \quad\quad \quad\quad(\ref{the3.3}.5)$$
It is obvious that $\Phi(g)(x)\leq y_0$ and $\Phi(g)$ has a lower bounded $y_1$.

Let us show that $g(x)\leq\Phi(g)(x)$ for each $x\in \bigcup_{n\in\mathbb{N}}F(n, (U_j(g)))$.
Also there exists $n\in\mathbb{N} $ such that $x\in F(n+1,(U_j(g)))\setminus F(n,(U_j(g)))$.
Then $g(x)\leq y_{n+1}=\Phi(g)(x)$ by $x\in F(n+1,(U_j(g)))\subset U_{n+1}(g)=Int\{x\in X: g(x)\leq y_{n+1}\}$.
If $x\in F(1,(U_j(g)))$, it is obvious that $g(x)\leq y_{1}=\Phi(g)(x)$. Thus, we have $g\leq\Phi(g)$.

To show that  $\Phi(g):X\rightarrow P$ is locally upper bounded about $(y_i)_{i\in\mathbb{N}}$, let $x\in X$.
By (\ref{the3.3}.3), there exists $n\in\mathbb{N}$ such that $x\in Int F(n+1, (U_j(g)))\setminus Int F(n, (U_j(g)))$.
Consider the neighborhood $O_x=Int F(n+1, (U_j(g)))$ of $x$.
For each $x'\in O_x$, it follows from the definition of $\Phi(g)$ that $\Phi(g)(x')\leq y_{n+1}$.
This completes the
proof that $\Phi(g)$ is locally upper bounded about $(y_i)_{i\in\mathbb{N}}$.

Next we show $\Phi(g)$ is l.s.c..
For each $x\in \bigcup_{n\in\mathbb{N}} F(n, (U_j(g)))$, there exists some $m\in\mathbb{N}$ such that $x\in F(m+1, (U_j(g)))\setminus F(m, (U_j(g)))$.
We consider the neighborhood $O_x=X\setminus F(m, (U_j(g)))$ of $x$.
For each $x'\in O_x$, we can get $\Phi(g)(x')\geq y_{m+1}=\Phi(g)(x)$.
Therefore, $\bigwedge\Phi(g)(O_x)=y_{m+1}$ and $O_x\in \mathcal{N}_{x*}(\Phi(g))$.
This provides that $\Phi(g)(x)$ admits $\Phi(g)_*(x)$ because of (c) of Proposition \ref{prop2.7}.
We have
$$\Phi(g)_*(x)=\bigvee_{N\in\mathcal{N}_{x*}(\Phi(g))}\bigwedge\Phi(g)(N)\geq\bigwedge_{x'\in O_x}\Phi(g)(x')=y_{m+1}=\Phi(g)(x).$$
Hence, $\Phi(g)_*(x)=\Phi(g)(x)$ for each $x\in X$, that is, $\Phi(g)(x)$ is l.s.c..

Finally, let $g':X\rightarrow P$ be a map with $g\leq g'$.
Then
$$\{x\in X: g'(x)\leq y_n\}\subseteq \{x\in X: g(x) \leq y_n\}$$
and hence, $U_n(g)\supseteq U_n(g')$ for each $n\in\mathbb{N}$.
Therefore we have $F(n,(U_j(g)))\supseteq F(n,(U_j(g')))$ for each $n\in\mathbb{N}$, and

$$\bigcup_{n\in\mathbb{N}}Int F(n, (U_j(g)))=\bigcup_{n\in \mathbb{N}}U_n(g)$$
$$\bigcup_{n\in\mathbb{N}}Int F(n, (U_j(g')))=\bigcup_{n\in \mathbb{N}}U_n(g')$$

For each $x\in \bigcup_{n\in \mathbb{N}}U_n(g')$,
there exists $n\in \mathbb{N}$ such that $x\in F(n+1, (U_j(g')))\setminus F(n, (U_j(g')))$.
That is $x\in F(n+1, (U_j(g')))\subseteq F(n+1, (U_j(g)))$.
By (\ref{the3.3}.5), we can get $\Phi(g)(x)\leq y_{n+1}=\Phi(g')(x)$,
which proves the necessity.

Conversely, let $(U_j)_{j\in\mathbb{N}}$ be an increasing open cover of $X$.
Define a map $g_{((U_j))}:X\rightarrow P$ by:

$$g_{((U_j))}(x)=\left\{
\begin{array}{rcl}
y_1     &      & {x\in U_{1}}\\
y_{n+1}    &      & {x\in U_{n+1}\backslash \texttt{}U_{n}}\\
\end{array} \right. \quad\quad \quad\quad(\ref{the3.3}.1)$$
Then, $g_{((U_j))}$ is locally upper bounded. In fact, for each $x\in X$ there exists $n\in \mathbb{N}$ such that $x\in U_{n+1}\setminus U_n$.
Consider the neighborhood $O_x=U_{n+1}$ of $x$.
For each $x'\in O_x$, it follows from the definition of $g_{((U_j))}$ that $g_{((U_j))}(x')\leq y_{n+1}$.
This shows that $g_{((U_j))}$ is locally upper bounded.

By the assumption, there exist an operators $\Phi$ assigning to each locally upper bounded map $g:X\rightarrow P$ with a lower bound,
a l.s.c. map $\Phi(g):X\rightarrow P$ which is locally upper bounded about $(y_i)_{j\in\mathbb{N}}$ such that
$g(x)\leq \Phi(g)$, and $\Phi(g)\leq \Phi(g')$ whenever $g\leq g'$.

For each sequence $(U_j)_{j\in\mathbb{N}}$ of increasing open cover of $X$, define

$$F(n,(U_j))=\{x\in X: \Phi(g_{((U_j))})(x)\leq y_n\}\quad\quad \quad\quad(\ref{the3.3}.2)$$
We can get that $F(n,(U_j))$ is closed, by (a) of Proposition \ref{prop2.7} and Proposition \ref{prop2.8}.
It suffices to show the operator $F$ satisfies (1), (2)' and (3) of Lemma \ref{lem2.5}.

To show (1) and (2), let $(U_j)_{j\in\mathbb{N}}$ be an increasing open cover of $X$.
To see $U_n \supseteq F(n,(U_j))$ for each $n\in\mathbb{N}$ and let $x\in F(n,(U_j))$.
Then, $g(x)\leq\Phi(g_{((U_j))})(x)\leq y_n$ and $g(x)\leq y_n$, and thus $x\in U_{m}\setminus U_{m-1}$ and $m\leq n$ by (\ref{the3.3}.1).
So we have $x\in U_m\subset U_n$.
Hence $U_n \supseteq F(n,(U_j))$ holds.
To show (2)', let $x\in X$, take a neighborhood $O_x$ of $x$ and $n\in\mathbb{N}$ such that $\Phi(g_{((U_j))}(x')\leq y_n$ for each $x'\in O_x$, that is, $x\in Int F(n,(U_j))$,
which shows that $\bigcup_{n\in\mathbb{N} }IntF(n,(U_j))=X$.

To show (3),
let $((G_j))$ be an increasing sequence of open subsets of $X$ such that $(U_j)\preceq (G_j)$.
Since $U_n\subseteq G_n$ for each $n\in \mathbb{N}$, it follows from (\ref{the3.3}.1) that $g_{((G_j))}(x)\leq g_{((U_j))}(x)$.
Hence, we have $\Phi(g_{((G_j))})\leq\Phi(g_{((U_j))})$.
Furthermore,
 $F(n,(U_j))=\{x\in X: \Phi(g_{((U_j))})(x)\leq y_n\}\subseteq \{x\in X: \Phi(g_{((G_j))})(x)\leq y_n\}=F(n,(G_j))$ for each $n\in\mathbb{N}$,
which implies that $F((U_j))\preceq F((G_j))$. This shows that $U$ satisfies (3) of Lemma \ref{lem2.5}.
So $X$ is MCP.
\end{proof}

\begin{theorem}\label{the3.4}
Let $P$ be a bi-bounded complete, dually continuous poset $P$ with a $\ll_{d}$-increasing $\ll$-limit point $y_0$.
Then $X$ is MCM if and only if
there exists an operators $\Phi$ assigning to each map $g:X\rightarrow P$ which is locally upper bounded about $(y_i)_{i\in\mathbb{N}}$ with a lower bound,
a l.s.c. map $\Phi(g):X\rightarrow P$ such that
$g\leq \Phi(g)\ll y_0$ and $\Phi(g)\leq \Phi(g')$ whenever $g\leq g'$.
\end{theorem}

\begin{proof}
The proof is obtained by a modification of that of Theorem \ref{the3.3}. So, we only show the outline of the proof.

Suppose that $X$ is MCM and $F$ is any operator that satisfies conditions (1), (2) and (3) of Lemma \ref{lem2.5}.
Let $g:X\rightarrow P$ be locally upper bounded about $(y_j)_{j\in\mathbb{N}}$.
For each $n\in \mathbb{N}$, we define
$$ U_n(g)=Int\{x\in X: g(x)\leq y_n\}  \quad\quad \quad\quad(\ref{the3.4}.1)$$
Then, $\{U_n(g):n\in\mathbb{N}\}$ is a increasing sequence of open subsets of $X$ because of Proposition \ref{prop2.9}.
It is clear that $\bigcup_{n\in \mathbb{N}}U_n(g)=X$.

Thus, we can define $\Phi(g):X\rightarrow P$ as follows:

$$\Phi(g)(x)=\left\{
\begin{array}{rcl}
y_1    &      & {x\in F(1,(U_j(g)))}\\
y_{n+1}    &      & x\in F(n+1,(U_j(g)))\setminus F(n,(U_j(g)))\\
\end{array} \right. \quad\quad \quad\quad(\ref{the3.4}.2)$$
It is obvious that $\Phi(g)(x)\ll y_0$ and $\Phi(g)$ has a lower bounded $y_1$.

Then, $\Phi(g):X\rightarrow P$ is l.s.c. such that $g\leq\Phi(g)$.
For a locally upper bounded map $g':X\rightarrow P$ with $g\leq g'$, we have $\Phi(g)\leq \Phi(g')$, which proves the necessity.

Conversely, let $(U_j)_{j\in\mathbb{N}}$ be an increasing open cover of $X$.
Define a map $g_{((U_j))}:X\rightarrow P$ by:

$$g_{((U_j))}(x)=\left\{
\begin{array}{rcl}
y_1     &      & {x\in U_{1}}\\
y_{n+1}    &      & {x\in U_{n+1}\backslash \texttt{}U_{n}}\\
\end{array} \right. \quad\quad \quad\quad(\ref{the3.4}.3)$$
Then, $g_{((U_j))}$ is locally upper bounded.

Set $F(n,(U_j))=\{x\in X: \Phi(g_{((U_j))})(x)\leq y_n\}$ for each $n\in\mathbb{N}$.
Then, $U_n \supseteq F(n,(U_j))$ for each $n\in\mathbb{N}$.
Now, let us show that $\bigcup_{n\in\mathbb{N} }F(n,(U_j))=X$.
Let $x\in X$. Since $\Phi(g)\ll y_0=\bigvee\{y_i:i\in\mathbb{N}\}$, and $\{y_i:i\in\mathbb{N}\}$ are directed, there exists $i\in\mathbb{N}$ such that $\Phi(g)(x)\leq y_i$.
That is, $x\in F(i,(U_j))$, which shows that $\bigcup_{n\in\mathbb{N} }F(n,(U_j))=X$.

For an increasing sequence $((G_j))$ of open subsets of $X$ such that $(U_j)\preceq (G_j)$, we have $F((U_j))\preceq F((G_j))$.
So $X$ is MCM.
\end{proof}

\section{Other results}
By analogy with Theorems 3.1 through 3.4, we can prove the following Theorems, which extend some earlier results.

\begin{theorem}
Let $P$ be a bi-bounded complete, dually continuous poset $P$ with a $\ll_{d}$-increasing $\ll$-limit point $y_0$.
Then $X$ is countably paracompact if and only if
there exists an operators $\Phi$ assigning to each map $g:X\rightarrow P$ which is locally upper bounded about $(y_i)_{i\in\mathbb{N}}$ with a lower bound,
a l.s.c. map $\Phi(g):X\rightarrow P$ which is locally upper bounded about $(y_i)_{i\in\mathbb{N}}$ such that
$g\leq \Phi(g)$.
\end{theorem}

\begin{theorem}
Let $P$ be a bi-bounded complete, dually continuous poset $P$ with a $\ll_{d}$-increasing $\ll$-limit point $y_0$.
Then $X$ is countably metacompact. if and only if
there exists an operators $\Phi$ assigning to each map $g:X\rightarrow P$ which is locally upper bounded about $(y_i)_{i\in\mathbb{N}}$ with a lower bound,
a l.s.c. map $\Phi(g):X\rightarrow P$ such that
$g\leq \Phi(g)\ll y_0$.
\end{theorem}

Recall that the stratifiable (semi-stratifiable)  spaces is the monotone versions of the perfectly normal (perfect) spaces.
We get the similar results for perfectly normal (perfect) spaces as follows.

\begin{theorem}
Let $P$ be a bi-bounded complete, continuous poset $P$ with a $\ll_{d}$-increasing $\ll$-limit point $y_0$.
Then $X$ is perfectly normal if and only if
there exists an operators $\Phi$ assigning to each map $g:X\rightarrow P$ with $F_g\neq\emptyset$,
a l.s.c. map $\Phi(g):X\rightarrow P$ such that
$g\leq \Phi(g)$ and that $\Phi(g)$ is locally upper bounded about $(y_i)_{i\in\mathbb{N}}$ at each $x\in U_g$.
\end{theorem}

\begin{theorem}
Let $P$ be a bi-bounded complete, continuous poset $P$ with a $\ll_{d}$-increasing $\ll$-limit point $y_0$.
Then $X$ is perfect if and only if
there exists an operators $\Phi$ assigning to each map $g:X\rightarrow P$ with an upper bound $y_0$,
a l.s.c. map $\Phi(g):X\rightarrow P$ such that
$g\leq \Phi(g)$  and that $\Phi(g)$ is upper bounded at each $x\in U_g$.
\end{theorem}

The proofs of Theorem 4.1, 4.2, 4.3 and 4.4 follow in the same way as Theorem 3.1, 3.2, 3.3 and 3.4.

T. Kubiak \cite[Theorem 2.5]{TK} and E. Lane and C. Pan (see \cite{LN} ) gave the characterizations of monotonically normal spaces by monotone insertions of real-valued functions.
From viewpoints of Theorem \ref{the3.3}, it is natural to ask about monotone
poset-valued insertions on monotonically normal and monotonically countably paracompact spaces.

\begin{definition}\cite{Yama}
A poset $P$ endowed with a topology is called sup-continuous if $\bigvee: J_2\rightarrow P$ is continuous, where $J_2=\{\langle x,y\rangle\in P\times P:\exists x\vee y\in P\}$ is endowed with the subspace topology of the product space $P \times P$.
Dually,
a poset $P$ endowed with a topology is called inf-continuous if $\bigwedge:M_2\rightarrow P$ is continuous, where $M_2=\{\langle x,y\rangle\in P\times P:\exists x\wedge y\in P\}$ is endowed with the subspace topology of the product space $P \times P$.
\end{definition}
A topological poset is a sup-and inf-continuous poset.
Now, define $J_1=M_1=P$ and $J_n=\{\langle x_1,x_2,\cdots, x_n\rangle\in P^n:\exists \bigvee_{i=1}^nx_i\in P\}$ and $M_n=\{\langle x_1,x_2,\cdots,x_n \rangle\in P^n:\exists \bigwedge_{i=1}^nx_i\in P\}$, endowed with the subspace topology of $P^n$, for each $n\in\mathbb{N}$ with $n\geq3$.

\begin{proposition}\label{prop4.6}\cite{Yama}
If $P$ is an upper-bounded (resp. lower-bounded) complete and sup-continuous (resp. inf-continuous) poset, then $\bigvee:J_n\rightarrow P$ (resp. $\bigwedge:M_n\rightarrow P$) is continuous for each $n\in\mathbb{N}$.
\end{proposition}

\begin{theorem}\label{the4.7}
Let $X$ be monotonically normal and monotonically countably paracompact and $P$ be a bi-bounded complete, dually continuous poset with a $\ll_{d}$-increasing $\ll$-limit point $y_0$
such that $y_i\ll_d y_{i+1}$, $y_i\ll y_0$ $(i\in \mathbb{N})$ and $y_0=\bigvee_{i\in \mathbb{N}}y_i$.
Let $g:X\rightarrow P$ be an u.s.c. map with a lower bound $\perp g$.
Assume that $g\leq y_0$, and $\langle g, y_0\rangle$ has interpolated points on $\{y_n; n\in \mathbb{N}\}$
such that $\langle g(x),y_0\rangle$ has interpolated points $y_{j}(x)$  and $y_{k}(x)$ on $\{y_n; n\in \mathbb{N}\}$ with $y_{j}(x)\leq y_{k}(x)$ and a monotone increasing path $\varphi_x:[0,1]\rightarrow [y_j(x),y_0]$
from some lower point of $y_k(x)$ for each $x\in X$ to $y_0$.
Then, there exist an operator $\Phi$ assigning to each u.s.c. map $g:X\rightarrow P$,
a continuous map $\Phi(g):X\rightarrow P$ such that
$g(x)\ll_{d}\Phi(g)(x)\ll y_0$ for each $x\in X$ and that
$\Phi(g)\geq \Phi(g')$ whenever $g\leq g'$.
\end{theorem}

\begin{proof}
There exists an operator $F$ satisfying (1), (2)' and (3) in Lemma \ref{lem2.5}.
Let $g:X\rightarrow P$ be a u.s.c. map where $g\leq y_0$, and $\langle g, y_0\rangle$ has interpolated points on $\{y_n: n\in \mathbb{N}\}$.
For each $n\in \mathbb{N}$, we define
$$ U_n=\{x\in X: g(x)\ll_dy_n\}  \quad\quad \quad\quad(\ref{the4.7}.1)$$
Then, $\{U_n:n\in\mathbb{N}\}$ is a increasing sequence of open subsets of $X$ because of Proposition \ref{prop2.9}.
It is clear that $\bigcup_{n\in \mathbb{N}}U_n=X$.
In fact, for each $x\in X$, there exists $n\in\mathbb{N}$ such that $g(x)\ll_d y_n\ll y_0$.
It provides that $x\in U_n$.

Hence, $F((U_j))=(F_n)_{n\in\mathbb{N}}$ is a sequence of closed subsets of $X$ such that
$$F_n\subset U_n \text{~for each~} n\in\mathbb{N};$$
$$\bigcup_{n\in\mathbb{N}}Int F_n=X;$$
$$F_n\subset F_{n+1}, n\in\mathbb{N};$$
Similarly, $F((Int F_{n}))=(E_n)_{n\in\mathbb{N}}$ is a sequence of closed subsets of $X$ such that
$$E_n\subset Int F_{n} \text{~for each~} n\in\mathbb{N};$$
$$\bigcup_{n\in \mathbb{N}}IntE_n=\bigcup_{n\in\mathbb{N}}Int F(n, (U_j))=X;$$
$$E_n\subset E_{n+1}, n\in\mathbb{N};$$
And $F((Int E_{n}))=(L_n)_{n\in\mathbb{N}}$ is a sequence of closed subsets of $X$ such that
$$L_n\subset Int E_{n} \text{~for each~} n\in\mathbb{N};$$
$$\bigcup_{n\in \mathbb{N}}IntL_n=\bigcup_{n\in\mathbb{N}}Int E(n, (U_j))=X;$$
$$L_n\subset L_{n+1}, n\in\mathbb{N};$$
Let $H_n=U_n\setminus L_{n-1}$, $G_n=IntF_n\setminus E_{n-1}$ for each $n\in \mathbb{N}$ and $H_1=U_1$, $G_1=IntF_1$.
It is obvious that $\{H_n:n\in\mathbb{N}\}$ and $\{G_n:n\in\mathbb{N}\}$ are locally finite open covers of $X$ such that $\overline{G_n}\subset H_n\subset U_n$ for each $n\in\mathbb{N}$.

Since $x$ is monotone normal, take a continuous function $j_n:X\rightarrow [0,1]$ by
$$j_n(x)=\left\{
\begin{array}{rcl}
0     &      & {x\in \overline{G_n}}\\
1    &      & {x\in X\setminus H_n}\\
\end{array} \right. \text{~for each~} n\in\mathbb{N}.\quad\quad \quad\quad(\ref{the4.7}.2)$$

For each $n\in\mathbb{N}$, there exists $t_n\in P$ and a continuous monotone increasing map $\varphi_n:[0,1]\rightarrow [y_n,y_0]\subset P$ such that
$$ \varphi_n(0)=t_n,  \text{~and~} \varphi_n(1)=y_0, y_n\leq t_n\leq y_{n+1}.\quad\quad \quad\quad(\ref{the4.7}.3)$$
Define a continuous map $k_n:X\rightarrow[y_n,y_0]\subset P$ by
$$k_n=\varphi_n\circ j_n \text{~for each~} n\in\mathbb{N}.\quad\quad \quad\quad(\ref{the4.7}.4)$$
Thus, we can define a continuous map $\Phi(g):X\rightarrow P$ as follows:
$$\Phi(g)(x)=\bigwedge_{n\in\mathbb{N}}k_n(x)\quad\quad \quad\quad(\ref{the4.7}.5)$$
for each $x\in X$. Now, to show that $\Phi(g)$ is defined, we set $\delta_x=\{n\in \mathbb{N}:x\in H_n\}$ for each $x\in X$.
For each $n\notin \delta_x$, $k_n(x)=\varphi_n\circ j_n(x)=\varphi_n(1)=y_0$.
For each $n\in \delta_x$, $k_n(x)=\varphi_n\circ j_n(x)\geq y_n$ because of the range of $\varphi_n$ is $[y_n,y_0]$.
It follows from $H_n\subset U_n$ that
$$g(x)\ll_dy_n\leq k_n(x) \text{~for each~} n\in \delta_x.$$
This shows that $\{k_n(x):n\in \delta_x\}$ and $\{y_n:n\in \delta_x\}$ have an lower bound $g(x)$, thus $\bigwedge_{n\in\delta_x}k_n(x)$ and $\bigwedge_{n\in\delta_x}y_n$ exist.
Hence, by (1) of \cite[Proposition I-1.2 (3)]{Gi}, we obtain that
$$g(x)\ll_d\bigwedge_{n\in\delta_x}y_n\leq\bigwedge_{n\in\delta_x}k_n(x)=\bigwedge_{n\in\delta_x}k_n(x)\wedge y_0=\Phi(g)(x).$$
Therefore, $\Phi(g)$ is defined, and we also have $g\ll_d \Phi(g)$.

Next, to show $\Phi(g)\ll y_0$, let $x\in X$.
Since $\{G_n:n\in\mathbb{N}\}$ is a cover of $X$, take $n'\in\mathbb{N}$ so as to satisfy $x\in G_{n'}$.
Then, it follows from $G_{n'}\subset U_{n'}$ that
$$\Phi(g)(x)\leq k_{n'}(x)=\varphi_{n'}\circ j_{n'}(x)=\varphi_{n'}(0)=t(n')\leq y_{n'}\ll y_0.$$
Thus, $\Phi(g)\ll y_0$ holds, because of (2) of Lemma \ref{lem2.6}.

Finally, to show that $\Phi(g)$ is continuous. Let $x\in X$, and take a neighborhood  $O_x$ of $x$ and a finite subset $\delta'_x$ of $\mathbb{N}$ such that $O_x\bigcap H_n\neq \emptyset$ for each $n\in \delta'_x$.
Then, $\delta_y\subset \delta'_x$ for each $y\in\delta_y$.
Since $\Phi(g)(y)$ can be re-expressed as $\Phi(g)(y)=\bigwedge_{n\in\delta'_x}k_n(y)\wedge y_0=\bigwedge_{n\in\delta'_x}k_n(y)$ for each $y\in O_x$, we have that $\langle k_n(y)\rangle_{n\in\delta'_x}\in M_{|\delta'_x|}$ for each $y\in O_x$.
This means $(\Delta_{n\in\delta'_x}k_n)(O_x)\subset M_{|\delta'_x|}$, where $\Delta_{n\in\delta'_x}k_n$ is the diagonal of mappings $\{k_n:n\in\delta'_x\}$.
Hence, on $O_x$, $\Phi(g)$ is the composition $\bigwedge\circ(\Delta_{n\in\delta'_x}k_n)$ of $\Delta_{n\in\delta'_x}k_n$ and $\bigwedge:M_{|\delta'_x|}\rightarrow P$.
By Proposition \ref{prop4.6}, $\Phi(g)$ is continuous at $x$.

Finally, let $g':X\rightarrow P$ be a map with $g\leq g'$.
Then
$$\{x\in X: g'(x)\ll_dy_n\}\subseteq \{x\in X: g(x)\ll_dy_n\}$$
and hence, $U_n\supseteq U'_n$ for each $n\in\mathbb{N}$.
Therefore we have $F_n\supseteq F'_n$, $E_n\supseteq E'_n$ and $L_n\supseteq L'_n$ for each $n\in\mathbb{N}$, and
$$\bigcup_{n\in \mathbb{N}}U'_n=\bigcup_{n\in\mathbb{N}}Int F'_n=\bigcup_{n\in\mathbb{N}}Int E'_n=\bigcup_{n\in\mathbb{N}}Int L'_n$$

For each $x\in X$, there exists $n\in \mathbb{N}$ such that $x\in G'_n=IntF'_n\setminus E'_{n-1}$.
That is $x\in IntF'_n\subseteq IntF_n$.
By (\ref{the4.7}.2), (\ref{the4.7}.3), (\ref{the4.7}.4) and the monotonicity of $\varphi_n(x)$ we can get $k_n(x)\geq t_n=\varphi'_n(0)=\varphi'_n\circ j'_n(x)=k'_n(x)$.
This implies that $\Phi(g)\geq \Phi(g')$ whenever $g\leq g'$.
This completes the proof.
\end{proof}

For a map $f:$ $X\rightarrow P$ and a point $y\in P$, we define
$$U_{f,y}=\{x\in X: \langle f(x),y\rangle \text{~has an interpolated point~}\}$$ and $U_{y,f}=\{x\in X: \langle y,f(x)\rangle\text{~has an interpolated point~}\}$.
We have known that monotonically normal and monotonically countably paracompact spaces are stratifiable, then we have the following question:

\begin{question}
Let $X$ be stratifiable and $P$ be a bi-bounded complete, dually continuous poset with a $\ll_{d}$-increasing $\ll$-limit point $y_0$
such that $y_i\ll_d y_{i+1}$, $y_i\ll y_0$ $(i\in \mathbb{N})$ and $y_0=\bigvee_{i\in \mathbb{N}}y_i$.
Do the following conditions exist?

Let $g:X\rightarrow P$ be an u.s.c. map with a lower bound $\perp g$.
Assume that $g\leq y_0$, and $\langle g, y_0\rangle$ has interpolated points on $\{y_n; n\in \mathbb{N}\}$
such that $\langle g(x),y_0\rangle$ has interpolated points $y_{j}(x)$  and $y_{k}(x)$ on $\{y_n; n\in \mathbb{N}\}$ with $y_{j}(x)\leq y_{k}(x)$ and a monotone increasing path $\varphi_x:[0,1]\rightarrow [y_j(x),y_0]$
from some lower point of $y_k(x)$ for each $x\in X$ to $y_0$.
Then, there exist an operator $\Phi$ assigning to each u.s.c. map $g:X\rightarrow P$,
a continuous map $\Phi(g):X\rightarrow P$ such that $g(x)\leq\Phi(g)(x)\leq y_0$,
$g(x)\ll_{d}\Phi(g)(x)\ll y_0$ for each $x\in U_{g,y_0}$ and that
$\Phi(g)\geq \Phi(g')$ whenever $g\leq g'$.
\end{question}

\vskip0.9cm

\end{document}